\documentclass[onecolumn,11pt]{ieeeconf}

\IEEEoverridecommandlockouts                              
\usepackage{mathtools}
\usepackage{amsmath,amssymb,amsfonts}
\usepackage{booktabs}
\usepackage{cite}
\usepackage[usenames,dvipsnames]{xcolor}
\usepackage{flushend}

\renewcommand{\a}{\boldsymbol{a}}

\newcommand{\x}{\boldsymbol{x}}
\newcommand{\w}{\boldsymbol{w}}
\newcommand{\f}{\boldsymbol{f}}
\newcommand{\g}{\boldsymbol{g}}
\newcommand{\vphi}{\boldsymbol{\varphi}}
\newcommand{\m}{\boldsymbol{m}}
\newcommand{\bmu}{\boldsymbol{\mu}}
\newcommand{\bnu}{\boldsymbol{\nu}}

\newcommand{\R}{\mathbb{R}}

\newcommand{\Zp}{\mathbb{Z}_{\geq 0}}

\newcommand{\A}{\boldsymbol{A}}
\newcommand{\B}{\boldsymbol{B}}

\newcommand{\C}{\boldsymbol{C}}

\newcommand{\I}{\boldsymbol{I}}

\usepackage{comment}

\newif\ifshow
\showtrue
\ifshow

\else
\excludecomment{scratch}
\fi

\newtheorem{remark}{\bfseries Remark}

\newtheorem{theorem}{\bfseries Theorem}

\title{\LARGE \bf
Approximate moment dynamics for polynomial and \\ trigonometric stochastic systems
}

\author{Khem Raj Ghusinga$^{1,*}$, Mohammad Soltani$^{1,*}$, Andrew Lamperski$^{2}$, Sairaj Dhople$^{2}$, Abhyudai Singh$^{1}$
\thanks{$^{*}$ These authors contributed equally.}	
\thanks{$^{1}$Khem Raj Ghusinga, Mohammad Soltani, and Abhyudai Singh are with the Department of Electrical and Computer Engineering, University of Delaware, Newark, DE, USA 19716.
        {\tt\small \{khem,msoltani,absingh\}@udel.edu}}%
\thanks{$^{2}$Andrew Lamperski, and Sairaj Dhople are with the Department of Electrical and Computer Engineering, University of Minnesota, Minneapolis, MN, USA  55455.
        {\tt\small \{alampers,sdhople\}@umn.edu}}
}

\begin{document}

\maketitle
\thispagestyle{empty}
\pagestyle{empty}

\begin{abstract}
Stochastic dynamical systems often contain nonlinearities which make it hard to compute probability density functions or statistical moments of these systems. For the moment computations, nonlinearities in the dynamics lead to unclosed moment dynamics; in particular, the time evolution of a moment of a specific order may depend both on moments of order higher than it and on some nonlinear function of other moments. The moment closure techniques are used to find an approximate, close system of equations the moment dynamics. In this work, we extend a moment closure technique based on derivative matching that was originally proposed for polynomial stochastic systems with discrete states to continuous state stochastic systems to continuous state stochastic differential equations, with both polynomial and trigonometric nonlinearities. We validate the technique using two examples of nonlinear stochastic systems.
\end{abstract}


\section{Introduction}
Stochastic dynamical systems appear in numerous contexts in physics, engineering, finance, economics, and biology (see, e.g., \cite{allen07,lande03,malliaris82,gardiner86,oksendal03}). In terms of mathematical characterization, the most useful quantity in analysis of stochastic systems is the probability density function (pdf). However, the pdf is analytically intractable for most systems. So, numerical techniques, such as Monte Carlo simulation, are employed to compute the pdf \cite{hes05,jup09}. Generally speaking, in analysis of many stochastic systems, the goal is often less ambitious than computing the pdf, and knowing only a few lower order moments (mean, variance, etc.) might suffice.

If the system under consideration has polynomial dynamics, then time evolution of various statistical moments can be computed by solving a system of coupled linear differential equations. However, a major drawback of using these moment equations is that except for a few special cases such as systems with linear dynamics, the differential equations for moments up to a given order consist of terms involving higher-order moments. This is known as the problem of \emph{moment closure}. A typical way around this is to truncate the system of ODEs to a finite system of equations, and \emph{close} the moment equations using some sort of approximation for a given moment in terms of moments of lower order \cite{lkk09,sih10,gil09,svs15,jdd14,sih07ny}. If the system under consideration involves nonlinearities such as trigonometric functions that often arise in swing equations, then the differential equations describing the moments involve moments of nonlinear functions of the state. In such cases, usage of moment closure schemes is rather limited.

For systems with polynomial dynamics, a number of moment closure techniques have been proposed to approximate a higher order moment in terms of lower order moments. Some of these techniques make prior assumptions on the distribution of the system, while others attempt to find a linear or nonlinear approximation of the moment dynamics \cite{Kuehn16,socha2008linearization}. One method that falls in the latter category is the derivative matching based closure \cite{SinghHespanhaDM}. Here, a nonlinear approximation of a given moment is obtained in terms of lower order moments by matching the derivatives of the original moment dynamics with the proposed approximate dynamics at some initial point in time. This method was originally proposed for approximating moment dynamics of biochemical reaction systems which are described via discrete states \cite{SinghHespanhaDM}. Given the attention received by this approach and its superior performance than several moment closure schemes \cite{svs15,svn14}, we apply it to close moments for nonlinear stochastic systems described via stochastic differential equations (SDEs). We further extend the method to include trigonometric functions in the dynamics. Our results show that the derivative matching technique provides reasonably good approximation to the moment dynamics.

Remainder of the paper is organized as follows. In section II, we describe the moment equations for a stochastic differential equations, and discuss the moment closure problem. In section III, we discuss the the derivative matching moment closure technique for SDEs and provide a proof for it. We illustrate the technique via examples in section IV. The paper is concluded in section V, along with a few directions of future research. 

 {\it Notation:} Vectors and matrices are denoted in bold. The set of real numbers and non-negative integers are respectively denoted by $\R$ and $\Zp$. The expectation  is represented by angled-brackets, $\left< \right>$. $\I$ is used to denote the Identity matrix.

 \section{Moment Dynamics of an SDE}
Consider a $n$-dimensional stochastic differential equation (SDE) represented as
\begin{equation}
d\x = \f(\x,t) \,dt + \g(\x,t) \,d\w_t,
\label{eq:sde}
\end{equation}
where $\x = \begin{bmatrix} x_1 & x_2 & \ldots & x_n \end{bmatrix}^\top \in \R^n$ is the state vector; $\f(\x,t)=\begin{bmatrix} f_1(\x,t) & f_2(\x,t) & \ldots & f_n(\x,t) \end{bmatrix}^\top: \R^n \times [0,\infty) \rightarrow \R^n$ and $\g(\x,t)=\begin{bmatrix} g_1(\x,t) & g_2(\x,t) & \ldots & g_n(\x,t) \end{bmatrix}^\top:\R^n \times [0,\infty) \rightarrow \R^n$ describe the system dynamics; and $\w_t$ is the $n$-dimensional Weiner process satisfying
\begin{align}
\left<d \w_t \right>=\mathbf{0}, \quad \left<d\w_t \, d\w_t^\top\right>=\I \,dt,
\end{align}
where $\I$ is an $n\times n$ Identity matrix. We further assume that sufficient mathematical requirements for the existence of the solution to  \eqref{eq:sde} are satisfied (see, e.g., \cite{oksendal03}).

The moments of an SDE can be obtained using the well-known It\^o formula  \cite{oksendal03}. This formula states that for any smooth scalar-valued function $h(\x(t))$, the increment is given by
\begin{equation}\label{eq:ito}
dh(\x(t)) = \frac{\partial h(\x(t))}{\partial x} \left( f(\x(t)) dt + g(\x(t)) d \w (t) \right) \\
+\frac{1}{2} \text{ Tr } \left(\frac{\partial^2 h(\x(t))}{\partial x^2} g(\x(t)) g(\x(t))^\top\right)dt.
\end{equation}
Taking expectations and dividing both sides by $dt$ gives the following
differential equation
\begin{equation}
\label{eq:generatorDiffEq}
\frac{d}{dt} \langle h(\x(t)) \rangle
=
\left\langle \frac{\partial h(\x(t))}{\partial x} 
f(\x(t)) 
+\frac{1}{2} \text{ Tr } \left(
\frac{\partial^2 h(\x(t))}{\partial x^2} g(\x(t)) 
g(\x(t))^\top
\right) \right\rangle .
\end{equation}
Let $h(\x)$ be monomial of the form
\begin{equation}
h(\x)=x_1^{m_1} x_2^{m_2} \ldots x_n^{m_n}=:\x^{[\m]},
\end{equation}
where $\m={\begin{bmatrix}m_1 &  m_2 & \ldots & m_n \end{bmatrix}}^\top \in \Zp^n $, then $\left<h(x)\right>$ represents a moment of $\x$. For a given $\m$, we represent the moment by $\mu_{\m}=\left< \x^{[\m]}\right>$. Using  \eqref{eq:generatorDiffEq}, dynamics of $\mu_{\m}$ evolves according to 
\begin{equation}
\frac{d \mu_{\m}}{dt}=\sum_{i=1}^{n} \left<f_i \frac{\partial \x^{[\m]}}{\partial x_i}\right>
 +\frac{1}{2}\sum_{i=1}^{n}\sum_{j=1}^{n} \left<(\g \g^\top)_{ij} \frac{\partial^2 \x^{[\m]}}{\partial x_i \partial x_j}\right>.
 \label{eq:momentdyn}
\end{equation}
The sum $\sum_{j=1}^{n}m_j$ is referred to as the order of the moment.

As long as $\f(\x,t)$ and $\g(\x,t)$ are linear in $\x$, a moment of a certain order is a linear combination of other moments of same or smaller order \cite{socha2008linearization}. Hence, if we construct a vector $\bmu$ consisting of all moments up to the $M^{\rm th}$ order moments of $\x$, its time evolution is captured by the solution of the following system of linear differential equations:
\begin{equation}\label{eq:lin}
\frac{d\bmu}{dt}=\a+\A\bmu.
\end{equation}
Here, $\bmu={\begin{bmatrix}\mu_{\m_1} &  \mu_{\m_2} & \ldots &  \mu_{\m_k} \end{bmatrix}}^\top, \m_p \in \Zp^{n}, \forall p \in \{1,2,\ldots, k\}$ is assumed to be a vector of $k$ elements. The vector $\a$ and the matrix $\A$ are determined by the form of $\f(\x,t)$ and $\g(\x,t)$. Under some mild assumptions, standard tools from linear systems theory can be used to obtain solution to \eqref{eq:lin}, and it is given by
\begin{equation}
\bmu(t)=-\A^{-1}\a+e^{\A t} \left( \bmu(0)+\A^{-1}\a\right).
\end{equation} 

\begin{remark}\label{rem:mu}
It is easy to see that there are $(m+n-1)!/(m! (n-1)!)$ moments of order $m$.  Therefore, the dimension of the vector $\bmu$ in \eqref{eq:lin} is given by
\begin{equation}
k=\sum_{m=1}^{M} \frac{(m+n-1)!}{m! (n-1)!}= \frac{(M+n)!}{M! n!}-1.
\end{equation}
Without loss of generality, we can assume that the elements in $\bmu$ are stacked up in graded lexicographical order. That is, the first $n$ elements in $\bmu$ are the moments of first order, next $n(n+1)/2$ elements are moments of the second order, and so on. $\blacksquare$
\end{remark}

In general, when $\f(\x,t)$ and $\g(\x,t)$ are polynomials in $\x$, the time derivative of a moment might depend on moments of order higher than it. To see this, consider the following one dimensional cubic drift
\begin{equation}
dx = -  x^3 dt + dw_t.
\end{equation} 
The time evolution of a moment of order $m \geq 1$ is given by
\begin{align}
\frac{d\left<x^m\right>}{dt}&=\left<\frac{\partial x^m}{\partial x} (-x^3)\right> + \frac{1}{2}\left<\frac{\partial x^m}{\partial x^2}\right> \\
&=-m \left<x^{m+2}\right> + \frac{m(m-1)}{2}\left<x^{m-1}\right>,
\end{align}
which clearly depends upon the $(m+1)^{th}$ moment. In other words, the moment dynamics is not closed. Thus, for systems with nonlinear dynamics, the moment equations in \eqref{eq:lin} need to be modified to a general form
 \begin{equation}
\frac{d\bmu}{dt}=\a+\A\bmu + \B\overline{\bmu},
\label{eqn:openmomentdynamics}
\end{equation}
where $\overline{\bmu} \in \R^r $ is a vector of moments of order greater than or equal to $M+1$. 

The solution to \eqref{eqn:openmomentdynamics} is generally obtained by approximating the higher order moments in $\overline{\bmu}$ as, possibly nonlinear, functions of lower order moments in $\bmu$. The approximation might be made by assuming some underlying distribution, or by applying some other physical principle \cite{Kuehn16,socha2008linearization}. Essentially the moment closure methods translate to finding an  approximation of  \eqref{eqn:openmomentdynamics} by a system of equations 
\begin{subequations}
\begin{align}
& \frac{d\bnu}{dt}=\a+\A\bnu + \B\overline{\vphi}(\bnu), \\
& \bnu={\begin{bmatrix}\nu_{\m_1} & \nu_{\m_2} & \ldots &  \nu_{\m_k} \end{bmatrix}}^\top,
\label{eqn:closedapproxmomentdynamics}
 \end{align}
 \end{subequations}
where the function $\overline{\vphi}:\R^k \rightarrow \R^r $ is chosen such that $\bmu(t) \approx \bnu(t)$. Here, $M$ is called the order of truncation. 

If the functions $\f(\x,t)$ are not polynomials, then it may not be possible to obtain a convenient form like \eqref{eqn:openmomentdynamics} for the moments. For instance, consider the following differential equation
\begin{align}
dx_1&=x_2 dt \\
dx_2 &= - \sin x_1  dt + dw_t.
\end{align} 
Here, the time evolution of $\left<x_2\right>$ is given by
\begin{align}
\frac{d\left<x_2\right>}{dt}&=-\left<\sin x_1 \right>,
\end{align}
which depends a nonlinear moment $\left<\sin x_1\right>$. Although, \eqref{eq:generatorDiffEq} can be used to write the dynamics of $\left<\sin x_1\right>$, it will further depend on other trigonometric moments. In Section IV, we will consider  a system of this type and perform moment closure. In the next section, we first discuss the derivative matching closure scheme for SDEs.

\section{Derivative Matching Moment Closure Technique for SDEs}
In this section, we describe the derivative matching based moment closure technique for SDEs. As the name suggests, the closure is performed by matching time derivatives of $\bmu(t)$ and $\bnu(t)$. This technique was originally proposed for approximating moment dynamics of discrete--state continuous--time systems \cite{SinghHespanhaDM, SinghHespanhaLogNormal}. The derivative matching technique attempts to approximate $\bmu(t)$ by some $\bnu(t)$ such that a sufficiently large number of their derivatives match point-wise. The idea being that if the values of these two vectors at some time $t_0$ are equal, and their derivatives up to certain order also match, then they would closely follow each other for some time interval after $t_0$. More precisely, for each $\delta >0$ and $N \in \Zp$, $\exists\; T \in \R$ such that if 
\begin{equation}
\bmu(t_0)=\bnu(t_0) \implies \quad \frac{d^i \bmu(t)}{dt^i} \Bigg \lvert_{t=t_0}=\frac{d^i \bnu(t)}{dt^i} \Bigg \lvert_{t=t_0},
\label{eq:dm}
\end{equation}
hold for a $t_0 \in [0,\infty)$ and $i={1,2,\ldots,N}$, then
\begin{equation}
\|\bmu(t)-\bnu(t)\| \leq \delta, \quad \forall t \in [t_0, T].
\label{eq:munuclose}
\end{equation}
Further, one can obtain the bound in \eqref{eq:munuclose} for the interval $[t_0,\infty)$ under some appropriate asymptotic conditions \cite{hespanha2005polynomial}.

To construct the closed moment dynamics, we follow similar steps as \cite{SinghHespanhaDM}. Consider a vector $\overline{\m} \in \Zp^n$ such that $\mu_{\overline{\m}}$ is an element in $\overline{\bmu}$. We approximate $\mu_{\overline{\m}}$ as a function of elements in the vector $\bmu$. Denoting the corresponding approximation of $\mu_{\overline{\m}}$ in $\overline{\vphi}(\bmu)$ by $\phi_{\overline{\m}}(\bmu)$, the following separable form is considered
\begin{equation}
\phi_{\overline{\m}}(\bmu)=\prod_{p=1}^{k}\left(\mu_{\m_p}\right)^{\alpha_p},
\label{eq:sdmmc}
\end{equation}
where $\alpha_p$ are appropriately chosen constants. Generally speaking, \eqref{eq:dm} is a strong requirement and it is not possible to find the coefficients $\alpha_p$ such that it holds for every initial condition. We, therefore, consider a relaxation of this by seeking  $\alpha_p$ such that the derivatives match for a deterministic initial condition $\x(t_0)=\x_0$. Next, we state a theorem showing that the coefficients $\alpha_p$ can be obtained by solving a system of linear equations. Before that, we define a short-hand notation that is used in the theorem. 
For two vectors $\hat{\m}={\begin{bmatrix}
	\hat{m}_1 & \hat{m}_2 & \ldots & \hat{m}_n
	\end{bmatrix}}^\top \in \Zp^{n}$ and $\breve{\m}={\begin{bmatrix}
	\breve{m}_1 & \breve{m}_2 & \ldots & \breve{m}_n
	\end{bmatrix}}^\top \in  \Zp^{n}$,  we have the following notation
\begin{subequations}
	\begin{equation}
	\C_{\breve{\m}}^{\hat{\m}}:= C_{\breve{m}_1}^{\hat{m}_1}C_{\breve{m}_2}^{\hat{m}_2}\cdots C_{\breve{m}_n}^{\hat{m}_n},
	\label{eq:boldCdef}
	\end{equation}
	where 
	\begin{equation}
	C_{l}^{h}=
	\begin{cases}
	\frac{h!}{l! (h-l)!}, h \geq l,\\
	0, h < l.
	\end{cases}
	\end{equation}
\end{subequations}

\begin{theorem}
For each element $\mu_{\overline{\m}}$ of the vector $\overline{\bmu}$, assume that the corresponding moment closure function $\phi_{\overline{\m}(\bmu)}$ in the vector $\overline{\vphi}(\bmu)$ is chosen according to  \eqref{eq:sdmmc} with the coefficients $\alpha_p$ chosen as the unique solution to the following system of linear equations
\begin{equation} \label{eqn:coeff}
\C_{[\m_s]}^{[\overline{\m}]}=\sum_{p=1}^{k} \alpha_p \C_{[\m_s]}^{[\m_p]}, \quad s=1,2,\cdots,k.
\end{equation}
Then, for every initial condition $\x(t_0)=\x_0 \in \R^n$, we have that
\begin{subequations}
\begin{align}
 \bmu(t_0)=\bnu(t_0) & \implies \frac{d \bmu(t)}{dt}\Bigg|_{t=t_0}=\frac{d \bnu(t)}{dt}\Bigg|_{t=t_0} \\
& \implies \frac{d^2 \bmu(t)}{dt^2}\Bigg|_{t=t_0}=\frac{d^2 \bnu(t)}{dt^2}\Bigg|_{t=t_0}.
\end{align}
\end{subequations}
\end{theorem}

\begin{proof}
It is sufficient to prove that for each element $\mu_{\overline{\m}}$ of $\overline{\bmu}$ and its corresponding moment closure function $\phi_{\overline{\m}}(\bmu)$, we have the following:
\begin{subequations}
\begin{align}
\mu_{\overline{\m}}(t_0)&=\phi_{\overline{\m}}(\bmu(t_0)), \label{eq:pf1} \\
\frac{d\mu_{\overline{\m}(t)}}{dt} \Bigg \lvert_{t=t_0}&=\frac{d\phi_{\overline{\m}}(\bmu(t))}{dt} \Bigg \lvert_{t=t_0}.\label{eq:pf2}
\end{align}
\end{subequations}
We first show that  \eqref{eq:pf1} holds. Since initial conditions are $\x(t_0)=\x_0$ with probability one, we have
\begin{subequations}
\begin{align}
\mu_{\overline{\m}}(t_0)&=\x_0^{[\overline{\m}]}, \label{eq:pf11} \\
\phi_{\overline{\m}}(\bmu(t_0))&=\prod_{p=1}^{k}\left(\x_0^{[\m_p]}\right)^{\alpha_p}=\x_0^{[\sum_{p=1}^{k}\alpha_p \m_p]}.\label{eq:pf12}
\end{align}
\end{subequations}
Recall Remark \ref{rem:mu}, that without loss of generality, the moments in vector $\bmu$ can be assumed to be stacked in graded lexicographical order. Thus, the first $n$ elements of $\bmu$ are moments of order one. This allows us to write
\begin{subequations}
\begin{align}
\overline{\m}&={\begin{bmatrix}
\C_{{\m_1}}^{\overline{\m}},\C_{{\m_2}}^{\overline{\m}}, \ldots,\C_{{\m_n}}^{\overline{\m}}
\end{bmatrix}}^\top, \label{eq:mbar}\\
\m_p&={\begin{bmatrix}
\C_{{\m_1}}^{\m_p},\C_{{\m_2}}^{\m_p}, \ldots,\C_{{\m_n}}^{\m_p}
\end{bmatrix}}^\top, \forall p={1,2,\ldots,k}, \label{eq:mp}
\end{align}
\end{subequations}
where a vector $\m_i \in \Zp^n, i=1,2,\ldots,n$ has $1$ at the $i^{th}$ position, and rest of the elements are zero. Using these relations, and  \eqref{eq:boldCdef} for $s=1,2,\ldots,n$, we obtain 
\begin{equation}
\overline{\m}=\sum_{p=1}^{k}\alpha_p \m_p.
\label{eq:O1relation}
\end{equation}
Substituting this result in \eqref{eq:pf11} proves  \eqref{eq:pf1}.

Next, we prove that \eqref{eq:pf2} holds. For this part, we assume that $\x_0={\begin{bmatrix}x_{01}, x_{02}, \ldots, x_{0n}\end{bmatrix}}^\top \in \R^n$. Consider
\begin{subequations}
\begin{align}
&\frac{d{\phi}_{\overline{\m}}(\bmu(t))}{dt}\Bigg |_{t=t_0} \\
& \qquad = \phi_{\overline{\m}}(\bmu(t_0)) \sum_{p=1}^{k}\alpha_p \frac{\frac{d\mu_{\m_p}(t)}{dt}\Big |_{t=t_0}}{\mu_{\m_p}(t_0)} \\
& \qquad = \sum_{p=1}^{k}\alpha_p \x_0^{[\overline{\m}-\m_p]}\frac{d\mu_{\m_p}(t)}{dt}\Bigg |_{t=t_0}.
\end{align}
\end{subequations}
Assuming $\m_p ={\begin{bmatrix}m_{p1} & m_{p2} & \ldots & m_{pn}\end{bmatrix}}^\top  \in \Zp^{n}$, we can use  \eqref{eq:momentdyn} to obtain the expression for $ \displaystyle \frac{d\mu_{\m_p}(t)}{dt}=\frac{d\left<\x^{[\m_p]}\right>}{dt}$. This enables us to write
\begin{subequations}
\small{\begin{align}
\frac{d \phi_{\overline{\m}}(\bmu(t))}{dt}\Bigg |_{t=t_0} 
&=\sum_{p=1}^{k}\alpha_p \x_0^{[\overline{\m}]} \sum_{i=1}^{n} \frac{m_{pi}}{x_{0i}}f_i(\x_0,t_0)+ \frac{1}{2}\sum_{p=1}^{k}\alpha_p \x_0^{[\overline{\m}]}\sum_{i=1}^{n}\frac{m_{pi}(m_{pi}-1)}{x_{0i}^2}\left(g(\x_0,t_0)g^\top(\x_0,t_0)\right)_{ii}\nonumber \\
& \qquad  + \frac{1}{2}\sum_{p=1}^{k}\alpha_p \x_0^{[\overline{\m}]}\sum_{\substack{i,j=1\\ i\ne j}}^{n}\frac{m_{pi} m_{pj}}{x_{0i}x_{0j}}\left(g(\x_0,t_0)g^\top(\x_0,t_0)\right)_{ij} \\
&  =\x_0^{[\overline{\m}]} \sum_{i=1}^{n} \frac{\sum_{p=1}^{k}\alpha_p m_{pi}}{x_{0i}}f_i(\x_0,t_0) + \frac{1}{2}\x_0^{[\overline{\m}]}\sum_{i=1}^{n}\frac{\sum_{p=1}^{k}\alpha_p m_{pi}(m_{pi}-1)}{x_{0i}^2}\left(g(\x_0,t_0)g^\top(\x_0,t_0)\right)_{ii}\nonumber \\
& \qquad + \frac{1}{2}\x_0^{[\overline{\m}]}\sum_{\substack{i,j=1\\ i\ne j}}^{n}\frac{\sum_{p=1}^{k}\alpha_p m_{pi} m_{pj}}{x_{0i}x_{0j}}\left(g(\x_0,t_0)g^\top(\x_0,t_0)\right)_{ij}.
\end{align}}
\end{subequations}
\normalsize

Comparing this with the expression for $\frac{d\mu_{\overline{\m}}}{dt}$ computed at $t=t_0$, which can be calculated from \eqref{eq:momentdyn} and assuming $\overline{\m}={\begin{bmatrix}\overline{m}_{1}, \overline{m}_{2}, \ldots, \overline{m}_{n}\end{bmatrix}}^\top  \in \Zp^n$, we require:
\begin{subequations}
\begin{align}
\sum_{p=1}^{k}\alpha_p m_{pi}&=\overline{m}_i, \label{eq:pf21}\\
\sum_{p=1}^{k}\alpha_p\frac{ m_{pi}(m_{pi}-1)}{2}&=\frac{\overline{m}_i(\overline{m}_i-1)}{2}, \label{eq:pf22}\\
\sum_{p=1}^{k}\alpha_p\frac{ m_{pi}m_{pj}}{2}&=\frac{\overline{m}_i \overline{m}_j}{2}.\label{eq:pf23}
\end{align}
\end{subequations}
Note that \eqref{eq:pf21} is nothing but the relation in \eqref{eq:O1relation} written element-wise. Further, we had assumed that the vector $\bmu$ has its elements stacked up in graded lexicographical order (Remark 1). In particular, the moments of second order start with the $(n+1)^{th}$ element. In that case, the equality in \eqref{eq:pf22} follows when relations in \eqref{eq:mbar}--\eqref{eq:mp} are used in \eqref{eq:boldCdef} for $s=n+1,2n+1,\cdots,n^2+1$ (i.e., the second order moments with one of the exponents as $2$ and rest of them as zeros). Likewise, \eqref{eq:pf23} holds for the rest of the second order moments wherein two exponents are $1$ and rest are zeros.
\end{proof}

\begin{remark}
It is worth noting that when the derivative--matching technique is applied for a discrete-state process, there is an error in matching the first two derivatives \cite{SinghHespanhaDM}. However, in case of a continuous state stochastic differential equation, the first two derivatives are matched exactly. Another important difference between the discrete state systems, and  continuous state systems is that in the latter, the first two derivatives are matched exactly regardless of the form of $\f$ and $\g$ whereas in the former,  one needs to assume polynomial form for the rates at which the states are changed. $\blacksquare$
\end{remark}

\begin{remark}
Although we do not have a proof, the solution to the system of linear equations in \eqref{eqn:coeff} results in integer values of the coefficients $\alpha_p$ for all examples we have solved thus far. $\blacksquare$
\end{remark}

\section{Numerical Validation}
In this section, we illustrate the derivative matching technique on two examples. The first example is a Van der Pol oscillator  that frequently arises in many engineering applications \cite{strogatz2014nonlinear}. In this case, the system dynamics consists of polynomial functions of the state vector. The second example is  a swinging pendulum subject to white noise. In this example, the dynamics consist of polynomial functions in one state and, and a trigonometric functions in another state. We show that the derivative matching technique can be straightforwardly applied to the second example.
\subsection{Van der Pol oscillator}
\begin{figure}[!b]
	\centering
	{\includegraphics[width=0.5\textwidth]{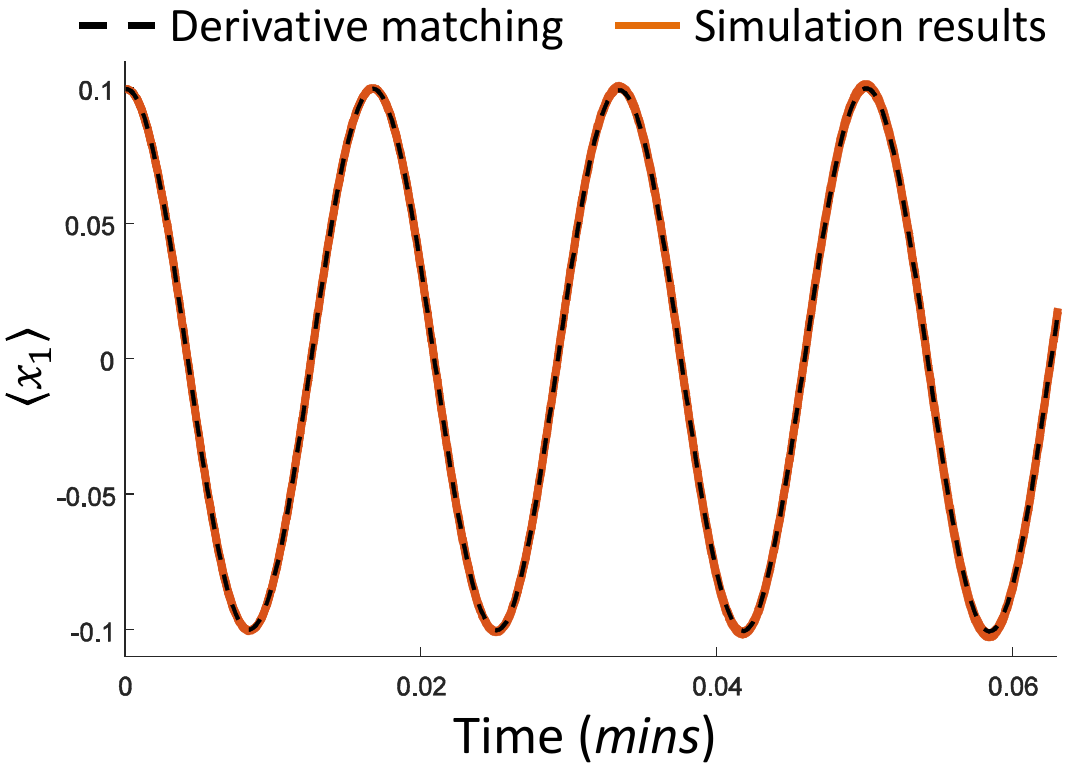}}
	\caption{Derivative matching technique replicates the oscillations of the Van der Pol oscillator quite reasonably. For this plot, the parameters values are $A=2.5$, $\omega_n =\omega_g=120\pi$, and $\epsilon=0.1$ . The initial conditions are taken as $x_1(0)=x_2(0)=0.1$.}
\end{figure}

In the deterministic setting, the Van der Pol oscillator is governed by the following second-order differential equation 
\begin{equation}
\frac{d^2 x}{d t^2 }-\epsilon  (1-x^2 )\frac{d x}{d t}+ \omega_n^2 x =A \cos(\omega_g t), 
\end{equation}
where $\epsilon$ is the bifurcation parameter, $\omega_n$ is the natural frequency, $\omega_g$ is the force frequency and $A$ is the force amplitude. A possible stochastic description of the oscillator could be to assume that the force is noisy, i.e., the actuators that apply the force also add a zero mean noise to the system. By choosing $x_1=x $ and $x_2 = \frac{d x }{d t}$, the oscillator dynamics could be written as
\begin{subequations}
	\begin{align}
	dx_1  = & x_2 dt,\\
	dx_2 =&\left( \epsilon  (1-x_1^2 )x_2- \omega_n^2 x_1+A \cos(\omega_g t)\right) dt + A dw_t. 
	\end{align}
\end{subequations} 
Suppose we are interested in the dynamics of $\left<x_1\right>$. To this end, we  write moment dynamics of this oscillator up to order two
\begin{subequations}
	\begin{align}	
	&	\frac{d\langle x_1 \rangle }{dt} =  \langle x_2 \rangle ,\\
	& 	\frac{d\langle x_2 \rangle }{dt} = \epsilon  (\langle x_2 \rangle -\langle x_1^2 x_2 \rangle )- \omega_n^2 \langle  x_1 \rangle +A \cos(\omega_g t) ,\\
	& 	\frac{d\langle x_1^2\rangle }{dt} = 2 \langle  x_1 x_2  \rangle ,\\
	& 	\frac{d\langle x_2^2 \rangle }{dt} =2  \epsilon  (\langle x_2^2  \rangle -\langle x_1^2 x_2^2 \rangle )- 2 \omega_n^2 \langle  x_1 x_2 \rangle  +2A \langle x_2 \rangle  \cos(\omega_g t)+ A^2 ,\\
	& 	\frac{d\langle x_1 x_2 \rangle }{dt} =\langle x_1^2 x_2 \rangle+  \epsilon  (\langle x_1 x_2 \rangle -\langle x_1^3 x_2 \rangle )- \omega_n^2 \langle  x_1^2  \rangle  +A \langle  x_1  \rangle \cos(\omega_g t)  .
	\end{align}\label{xx0}
\end{subequations}
As expected, the nonlinearities in the dynamics manifest  in unclosed moment dynamics, and the moment equations up to order two depend upon third and fourth order moments. In terms of notations in \eqref{eqn:openmomentdynamics}, we have $\bmu=\begin{bmatrix} \left<x_1\right> & \left<x_2\right> & \left<x_1^2\right> & \left<x_1 x_2 \right> & \left<x_2^2\right>\end{bmatrix}^\top$, and $
\overline{\bmu} = \left[\langle x_1^2 x_2 \rangle  \ \langle x_1^2 x_2^2 \rangle  \ \langle x_1^3 x_2 \rangle  \right]^\top$.

Applying the derivative matching closure as described in  Section III, we seek approximations of each element of $\overline{\bmu}$ in terms of those of $\bmu$ as in \eqref{eq:sdmmc}. Solving \eqref{eqn:coeff} for each of these yields the following approximations
\begin{subequations} \label{pendapprox}
	\begin{align}
	& \langle x_1^2 x_2 \rangle  \approx  \frac{\langle  x_1^2 \rangle \langle  x_1 x_2 \rangle^2 }{\langle x_1 \rangle^2 \langle x_2 \rangle  } , \\
	& \langle x_1^2 x_2^2 \rangle  \approx  \frac{\langle  x_1^2 \rangle \langle  x_1 x_2 \rangle^4 \langle  x_2^2 \rangle  }{\langle x_1 \rangle^4 \langle x_2 \rangle^4  } , \\
	& \langle x_1^3 x_2 \rangle  \approx  \frac{\langle  x_1^2 \rangle^3  \langle  x_1 x_2 \rangle^3}{\langle x_1 \rangle^6 \langle x_2 \rangle^2  } .
	\end{align}
\end{subequations}
Using the approximations from \eqref{pendapprox} in \eqref{xx0}, we obtain a closed set of moment equations. Fig. 1 compares the solution of $\left<x_1\right>$ with that of numerical simulations. Our results show an almost perfect match between the system with closure approximation and numerical simulations.

A caveat of the proposed derivative matching approximation is that, as in \eqref{pendapprox}, the mean of states appear in the denominator. Since the oscillator repeatedly crosses the zero, it is possible that some of these moments approach to zero. To avoid this, we add a small term $\delta$ to the denominator of approximations.

\subsection{Pendulum Swing}
\begin{figure}
	\centering
	\framebox{\includegraphics[width=0.4\textwidth]{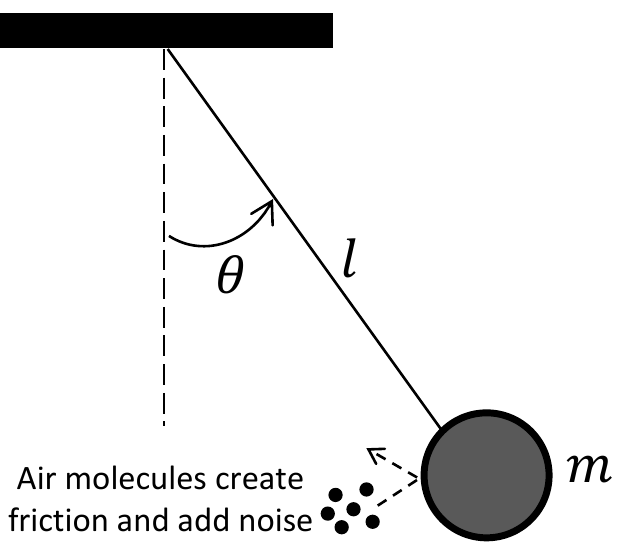}}
	\caption{Schematic of a pendulum interacting with air particles. }
	\label{fig:pendulum}
\end{figure}
In the deterministic setting, dynamics of a simple pendulum (see Fig.~\ref{fig:pendulum}) are given by
\begin{equation}
\frac{d^2 \theta}{d t^2 }+\frac{k}{m}\frac{d \theta}{d t} +\frac{g}{l} \sin \theta = 0
\end{equation}
where $g$ is the acceleration due to gravity, $l$ is the length of the pendulum, and $\theta$ is the angular displacement \cite{kha96}. We also consider friction in our system, with friction constant $k$. In the stochastic formulation, we could consider that the dynamics are affected by white noise that arises due random interaction of pendulum with air molecules. This term scales inversely with mass of the pendulum $m$, i.e., the interaction with gas air particles is negligible for a large mass. By choosing $x_1=\theta $ and $x_2 = \frac{d \theta }{d t}$,  the dynamics of the pendulum can be represented as
\begin{subequations}
	\begin{align}
	dx_1  = & x_2 dt,\\
	dx_2 =& \left(-\frac{k}{m} x_2 -\frac{g}{l}   \sin x_1 \right)dt + \frac{1}{m} dw_t.
	\end{align}\label{x}
\end{subequations} 
Here we have the trigonometric function $\sin x_1$, which gives rise to nonlinear behavior. To illustrate how derivative matching closure can be used in this context, we approximate $\left<\sin x_1\right>$ using \eqref{eq:generatorDiffEq}. To this end, we use Euler's relation to write
\begin{equation}
\sin x_1 = \frac{{\rm e}^{j x_1}-{\rm e}^{-j x_1}}{2j}. 
\end{equation}
With a change of variables, we can use the It\^o formula to transform \eqref{x} to the following
\begin{subequations}
	\begin{align}
	&	d{\rm e}^{j x_1} =  j{\rm e}^{j x_1} x_2dt,\\
	&	d{\rm e}^{-j x_1} =  -j{\rm e}^{-j x_1} x_2dt,\\
	&	dx_2 = \left(-\frac{k }{m} x_2 +\frac{j}{2} \frac{g}{l}{\rm e}^{j x_1} - \frac{j}{2} \frac{g}{l}{\rm e}^{-j x_1}  \right)dt + \frac{1}{m} dw_t.
	\end{align}\label{xx}
\end{subequations}
Fo these dynamics, we can write the moment dynamics with moments of $x_2$ appearing in  the form of monomials, and moments of $x_1$ appearing in the form of complex exponentials as below
\begin{subequations} 
	\begin{align}	
	&	\frac{d\langle {\rm e}^{j x_1}\rangle }{dt} =  j\left\langle {\rm e}^{j x_1 } x_2 \right\rangle ,\\
	& 	\frac{d\langle {\rm e}^{-j x_1 }\rangle }{dt} = -j\left\langle {\rm e}^{-j x_1 } x_2 \right\rangle ,\\
	&	\frac{d\langle x_2 \rangle }{dt} = -\frac{k}{m} \langle x_2 \rangle +\frac{j}{2} \frac{g}{l}\langle {\rm e}^{j x_1 } \rangle - \frac{j}{2} \frac{g}{l}\langle {\rm e}^{-j x_1 } \rangle,
\end{align}	
	\begin{align}	
	&	\frac{d\langle {\rm e}^{j x_1 }x_2\rangle }{dt} =  j\left\langle {\rm e}^{j x_1 } x_2^2 \right\rangle 
	-\frac{k}{m} \langle{\rm e}^{j x_1 }  x_2 \rangle+\frac{j}{2} \frac{g}{l}\langle {\rm e}^{2j x_1 } \rangle - \frac{j}{2} \frac{g}{l} , \\
	&	\frac{d\langle {\rm e}^{-j x_1 }x_2\rangle }{dt} =  -j\left\langle {\rm e}^{-j x_1 } x_2^2 \right\rangle 
	-\frac{k}{m} \langle{\rm e}^{-j x_1 }  x_2 \rangle -\frac{j}{2} \frac{g}{l}\langle {\rm e}^{-2j x_1 } \rangle + \frac{j}{2} \frac{g}{l} , \\ 
	& \frac{d\langle x_2^2 \rangle }{dt} = -2\frac{k}{m} \langle x_2^2 \rangle +j \frac{g}{l}\langle {\rm e}^{j x_1 }x_2  \rangle -j \frac{g}{l}\langle {\rm e}^{-j x_1 } x_2  \rangle +\frac{1}{m^2}, \\ &
	\frac{d\langle {\rm e}^{2j x_1 }\rangle }{dt} =  2j\left\langle {\rm e}^{2j x_1 } x_2 \right\rangle ,\\ 
	& \frac{d\langle {\rm e}^{-2j x_1 }\rangle }{dt} =  -2j\left\langle {\rm e}^{-2j x_1 } x_2 \right\rangle .
	\end{align}\label{expdyn}
\end{subequations}

One way to interpret the above mixed complex exponential monomial moment dynamics is to think that since all moments of $x_2$ are generated by taking expectations of the monomials $1, x_2, x_2^2, \ldots$,  we could consider the terms ${\rm e}^{j x_1}$ and ${\rm e}^{-j x_1}$ as two different variables. The mixed moments can then be generated by taking expectation of the products of the complex exponentials $1, {\rm e}^{-jx_1}, {\rm e}^{-2j x_1}, \ldots$ (or $1, {\rm e}^{jx_1}, {\rm e}^{2j x_1}, \ldots$) with the monomials $1, x_2, x_2^2, \ldots$. The order of the mixed moment can be thought of as the sum of powers of the monomials and complex exponentials. 

Given the above interpretation, the moment dynamics in \eqref{expdyn} are not closed. As per notation in \eqref{eqn:openmomentdynamics}, we have $\bmu=\left[
\left<e^{jx_1}\right> \ \left<e^{-jx_1}\right> \ \left<x_2\right> \ \ldots \left<x_2^2\right>  \right]^\top$, and $\overline{\bmu} = \left[\left\langle {\rm e}^{j x_1 } x_2^2 \right\rangle \  \left\langle {\rm e}^{-j x_1 } x_2^2 \right\rangle  \ \left\langle {\rm e}^{2j x_1 } x_2 \right\rangle \ \left\langle {\rm e}^{-2j x_1 } x_2 \right\rangle \right]^\top$. An important point to note is that since $e^{-jx_1} e^{j x_1}=1$, there is no need to consider their cross-moments. Thus, we only consider cross moments of $e^{-jx_1}$ with $x_2$, and $e^{jx_1}$ with $x_2$.

Next, we present different closure schemes for approximating moments in $\overline{\bmu} $ as nonlinear functions of moments up to order 2. As an example, consider the third-order moment $\left\langle {\rm e}^{j x_1 } x_2^2 \right\rangle$. The aim of closure is to approximate this moment as 
\begin{equation} 
\left\langle {\rm e}^{j x_1 } x_2^2 \right\rangle  \approx  \left\langle {\rm e}^{j x_1 }\right\rangle^{\alpha_1}   \left\langle {\rm e}^{j x_1 } x_2 \right\rangle^{\alpha_2}   \left\langle  x_2 \right\rangle^{\alpha_3}   \left\langle  x_2^2 \right\rangle^{\alpha_4}.
\end{equation}
Performing derivative matching approach as explained in Section III results in 	\begin{align}
\left\langle {\rm e}^{j x_1 } x_2^2 \right\rangle \approx  \frac{\left\langle  x_2^2 \right\rangle}{ \left\langle {\rm e}^{j x_1 }\right\rangle } \frac{\left\langle {\rm e}^{j x_1 } x_2 \right\rangle^2 }{\left\langle  x_2 \right\rangle^2 }.
\end{align}
With a similar approach we can approximate the other moments in the vector $\overline{\bmu}$
\begin{subequations}
	\begin{align}
	& \left\langle {\rm e}^{-j x_1 } x_2^2 \right\rangle \approx  \frac{\left\langle  x_2^2 \right\rangle}{ \left\langle {\rm e}^{-j x_1 }\right\rangle } \frac{\left\langle {\rm e}^{-j x_1 } x_2 \right\rangle^2 }{\left\langle  x_2 \right\rangle^2 }, \\
	& \left\langle {\rm e}^{2j x_1 } x_2 \right\rangle \approx  \frac{\left\langle   {\rm e}^{2j x_1 } \right\rangle}{ \left\langle x_2 \right\rangle } \frac{\left\langle  {\rm e}^{j x_1 } x_2 \right\rangle^2 }{\left\langle  {\rm e}^{j x_1 }  \right\rangle^2 },
	\\
	& \left\langle {\rm e}^{-2j x_1 } x_2 \right\rangle \approx  \frac{\left\langle   {\rm e}^{-2j x_1 } \right\rangle}{ \left\langle x_2 \right\rangle } \frac{\left\langle  {\rm e}^{-j x_1 } x_2 \right\rangle^2 }{\left\langle  {\rm e}^{-j x_1 }  \right\rangle^2 }.
	\end{align}
\end{subequations}
Another approximation that can be used is by assuming that the correlation in between two random variables is small due to presence of noise.  Hence the third order moment $\left\langle {\rm e}^{j x_1 } x_2^2 \right\rangle$ can be approximated as 
\begin{equation}
\left\langle {\rm e}^{j x_1} x_2^2\right\rangle \approx  \left\langle {\rm e}^{j x_1}\right\rangle     \left\langle  x_2^2\right\rangle .
\end{equation}
Similarly the rest of moments in $\overline{\bmu}$ can be approximated as
\begin{subequations}
	\begin{align}
	& \left\langle {\rm e}^{j x_1} x_2^2\right\rangle \approx  \left\langle {\rm e}^{j x_1}\right\rangle     \left\langle  x_2^2\right\rangle,\\
	& \left\langle {\rm e}^{-j x_1 } x_2^2 \right\rangle \approx   \left\langle {\rm e}^{-j x_1}\right\rangle     \left\langle  x_2^2\right\rangle, \\
	& \left\langle {\rm e}^{2j x_1 } x_2 \right\rangle \approx  \left\langle {\rm e}^{2 j x_1}\right\rangle     \left\langle  x_2\right\rangle ,
	\\
	& \left\langle {\rm e}^{-2j x_1 } x_2 \right\rangle \approx   \left\langle {\rm e}^{-2 j x_1}\right\rangle     \left\langle  x_2\right\rangle .
	\end{align}
\end{subequations}
The results of the closure approximations is compared to numerical solutions in Fig. 3. The results show that derivative matching provides reasonably accurate approximation of the moment dynamics.

\begin{figure}[h]
	\centering
	{\includegraphics[width=0.5\textwidth]{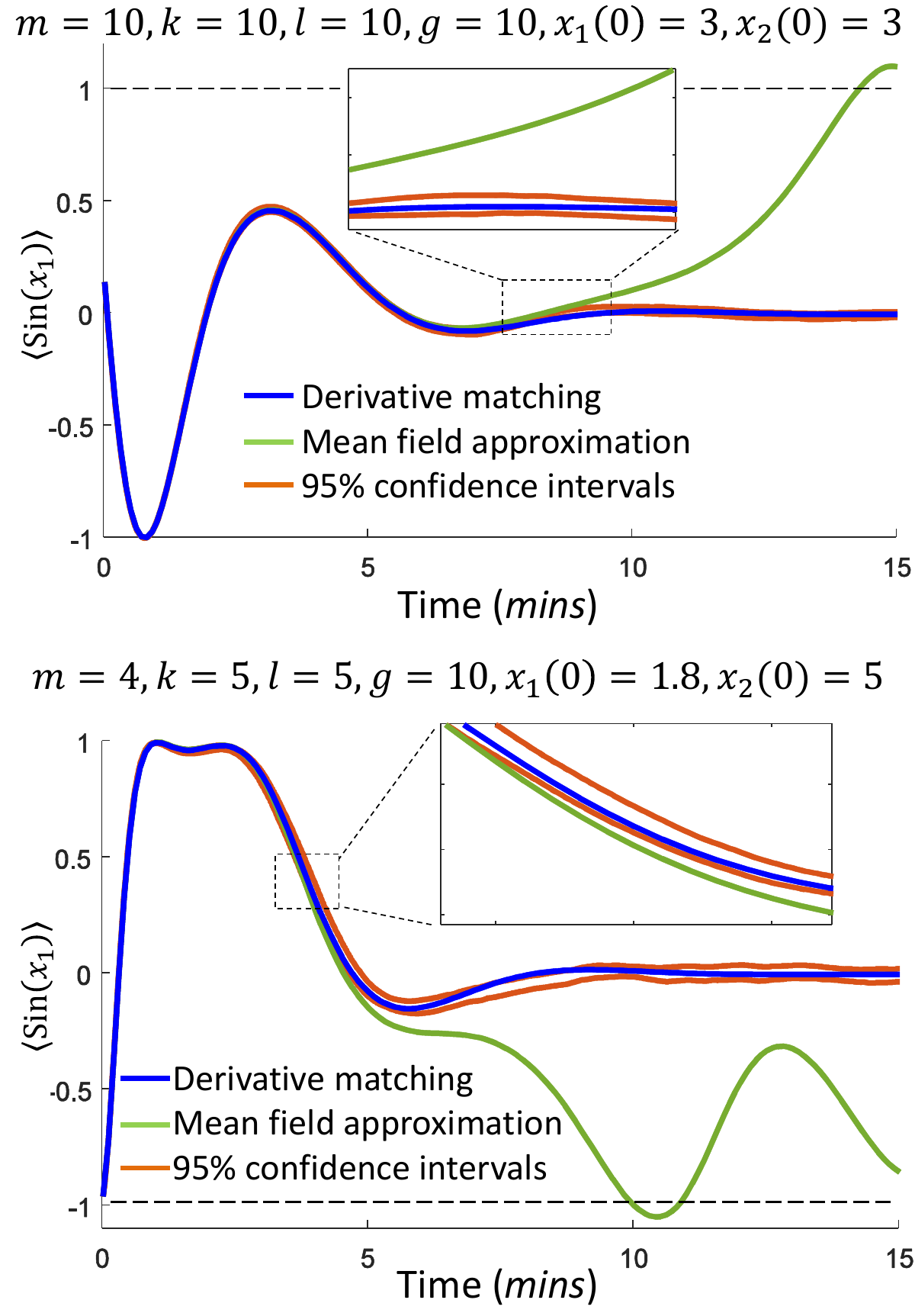}}
	\caption{Derivative Matching provides accurate approximation of the nonlinear function $\langle \sin(x_1) \rangle $. For comparison purpose,  $95\%$ confidence interval of the dynamics as obtained from numerical simulation, and approximate dynamics from a mean field approximation are shown.}	\label{figresults}
\end{figure}

\section{Conclusion}
In this paper, we extended the derivative matching based moment approximation method to stochastic dynamical systems with continuous state. We further illustrated that the method is not limited to polynomial dynamics, and it can be used to study systems that contain trigonometric functions. It would be interesting to extend the technique to other form of mixed functions, and also include differential algebraic inequalities. This would open possibilities of using the moment closure techniques to study a variety of nonlinearities, and has potential applications in power systems analysis. In addition, while in this paper we just considered continuous dynamics modeled through SDEs, many models contain both continuous dynamics and random discrete events \cite{hes06,tss14,hls00}. Deriving derivativ matching closure for such hybrid systems will be another avenue of research. Finally, we note that despite the promising results obtained by closure approximations, generally there are no guarantee on the errors of the closure approximation. Future work will carry out a detailed error analysis using other methods of finding bounds on moments  \cite{lamperski2017analysis}.

\section*{ACKNOWLEDGMENT}
AS is supported by the National Science Foundation Grant DMS-1312926, University of Delaware Research Foundation (UDRF) and Oak Ridge Associated Universities (ORAU).
\appendix
\section{}
\label{app:proof}
\bibliography{bibLoc,RefMaster}

\begin{thebibliography}{10}

\bibitem{allen07}
E.~Allen, {\em Modeling with It{\^o} stochastic differential equations},
  vol.~22.
\newblock Springer Science \& Business Media, 2007.

\bibitem{lande03}
R.~Lande, S.~Engen, and B.-E. Saether, {\em Stochastic population dynamics in
  ecology and conservation}.
\newblock Oxford University Press on Demand, 2003.

\bibitem{malliaris82}
A.~G. Malliaris, {\em Stochastic methods in economics and finance}, vol.~17.
\newblock North-Holland, 1982.

\bibitem{gardiner86}
C.~Gardiner, ``Handbook of stochastic methods for physics, chemistry and the
  natural sciences,'' {\em Applied Optics}, vol.~25, p.~3145, 1986.

\bibitem{oksendal03}
B.~{\O}ksendal, {\em Stochastic differential equations}.
\newblock Springer, 2003.

\bibitem{hes05}
J.~P. Hespanha, ``A model for stochastic hybrid systems with application to
  communication networks,'' {\em Nonlinear Analysis: Theory, Methods \&
  Applications}, vol.~62, pp.~1353--1383, 2005.

\bibitem{jup09}
A.~Julius and G.~Pappas, ``Approximations of stochastic hybrid systems,'' {\em
  IEEE Transactions on Automatic Control}, vol.~54, pp.~1193--1203, 2009.

\bibitem{lkk09}
C.~H. Lee, K.~Kim, and P.~Kim, ``A moment closure method for stochastic
  reaction networks,'' {\em Journal of Chemical Physics}, vol.~130, p.~134107,
  2009.

\bibitem{sih10}
A.~Singh and J.~P. Hespanha, ``Approximate moment dynamics for chemically
  reacting systems,'' {\em IEEE Transactions on Automatic Control}, vol.~56,
  pp.~414--418, 2011.

\bibitem{gil09}
C.~S. Gillespie, ``Moment closure approximations for mass-action models,'' {\em
  IET Systems Biology}, vol.~3, pp.~52--58, 2009.

\bibitem{svs15}
M.~Soltani, C.~A. Vargas-Garcia, and A.~Singh, ``Conditional moment closure
  schemes for studying stochastic dynamics of genetic circuits,'' {\em IEEE
  Transactions on Biomedical Systems and Circuits}, vol.~9, pp.~518--526, 2015.

\bibitem{jdd14}
J.~Zhang, L.~DeVille, S.~Dhople, and A.~Dominguez-Garcia, ``A maximum entropy
  approach to the moment closure problem for stochastic hybrid systems at
  equilibrium,'' in {\em Proc.~of the 53rd IEEE Conf.~on Decision and Control,
  Los Angeles, CA}, pp.~747--752, 2014.

\bibitem{sih07ny}
A.~Singh and J.~P. Hespanha, ``Stochastic analysis of gene regulatory networks
  using moment closure,'' in {\em Proc.~of the 2007 Amer. Control Conference,
  New York, NY}, 2006.

\bibitem{Kuehn16}
C.~Kuehn, {\em Moment Closure--A Brief Review}.
\newblock Understanding Complex Systems, Springer, 2016.

\bibitem{socha2008linearization}
L.~Socha, {\em Linearization Methods for Stochastic Dynamic Systems}.
\newblock Lecture Notes in Physics 730, Springer-Verlag, Berlin Heidelberg,
  2008.

\bibitem{SinghHespanhaDM}
A.~Singh and J.~P. Hespanha, ``Approximate moment dynamics for chemically
  reacting systems,'' {\em IEEE Transactions on Automatic Control}, vol.~56,
  no.~2, pp.~414--418, 2011.

\bibitem{svn14}
M.~Soltani, C.~A. Vargas-Garcia, N.~Kumar, R.~Kulkarni, and A.~Singh,
  ``Approximate statistical dynamics of a genetic feedback circuit,'' {\em
  Proc.~of the 2015 Amer. Control Conference, Chicago, IL}, pp.~4424--4429,
  2015.

\bibitem{SinghHespanhaLogNormal}
A.~Singh and J.~P. Hespanha, ``Lognormal moment closures for biochemical
  reactions,'' in {\em Proceedings of the 45th Conference on Decision and
  Control}, pp.~2063--2068, 2006.

\bibitem{hespanha2005polynomial}
J.~P. Hespanha, ``Polynomial stochastic hybrid systems,'' in {\em Hybrid
  Systems: Computation and Control}, pp.~322--338, 2005.

\bibitem{strogatz2014nonlinear}
S.~H. Strogatz, {\em Nonlinear dynamics and chaos: with applications to
  physics, biology, chemistry, and engineering}.
\newblock Westview press, 2014.

\bibitem{kha96}
H.~K. Khalil, {\em Nonlinear systems}, vol.~3.
\newblock Prentice Hall, NJ, 1996.

\bibitem{hes06}
J.~Hespanha, ``Modelling and analysis of stochastic hybrid systems,'' {\em IEE
  Proceedings Control Theory and Applications}, vol.~153, pp.~520--535, 2006.

\bibitem{tss14}
A.~R. Teel, A.~Subbaraman, and A.~Sferlazza, ``Stability analysis for
  stochastic hybrid systems: A survey,'' {\em Automatica}, vol.~50, no.~10,
  pp.~2435--2456, 2014.

\bibitem{hls00}
J.~Hu, J.~Lygeros, and S.~Sastry, ``Towards a theory of stochastic hybrid
  systems,'' in {\em Hybrid Systems: Computation and Control}, Lecture Notes in
  Computer Science, pp.~160--173, Springer, 2000.

\bibitem{lamperski2017analysis}
A.~Lamperski, K.~R. Ghusinga, and A.~Singh, ``Analysis and control of
  stochastic systems using semidefinite programming over moments,'' {\em arXiv
  preprint arXiv:1702.00422}, 2017.

\end{thebibliography}
\bibliographystyle{IEEEtran}
\end{document}